\documentclass[a4paper,12pt, reqno]{amsart}

\usepackage[cp1251]{inputenc}
\usepackage[ukrainian]{babel}
\usepackage{amssymb,upref}

\theoremstyle{plain}
\newtheorem{theorem}{Теорема}
\newtheorem{lemma}{Лема}

\theoremstyle{definition}
\newtheorem{definition}{Означення}
\theoremstyle{remark}

\uccode`ґ=`Ґ \uccode`Ґ=`Ґ \lccode`Ґ=`ґ \lccode`ґ=`ґ %
\uccode`є=`Є \uccode`Є=`Є \lccode`Є=`є \lccode`є=`є %
\uccode`і=`І \uccode`І=`І \lccode`І=`і \lccode`і=`і %
\uccode`ї=`Ї \uccode`Ї=`Ї \lccode`Ї=`ї \lccode`ї=`ї %
\begin{document}

УДК 517.5+517.17+517.18+511.72
\title[Функції зі складною локальною будовою] { Неперервні функції зі складною локальною будовою, визначені в термінах зображення чисел  знакопочережними рядами Кантора}

\author[С. О. Сербенюк]{С. О. Сербенюк}
\address{Інститут математики НАН України \\
         Київ, Україна}
\email{simon6@ukr.net; ~ simon.mathscience@imath.kiev.ua}

\begin{abstract}
Робота  присвячена одному нескінченнопараметричному класу неперервних функцій зі складною локальною будовою, визначених в термінах зображення чисел  знакопочережними рядами Кантора. Основна увага приділяється  диференціальним, інтегральним та іншим властивостям функцій. Знайдено умови монотонності та ніде немонотонності, вказано систему функціональних рівнянь, розв'язком якої є функція із заданого класу.

\textsc{Abstract.} The article is devoted to one infinite parametric class of continuous functions with complicated local structure such that  these functions are defined in terms of alternating Cantor series representation of numbers. The main attention is given to differential, integral and other properties of these functions. Conditions of monotony and nonmonotony  are discovered. The functional equations system, that the function from the given class of functions is a solution of the system is  indicated.
\end{abstract}

\maketitle
\begin{center}

{\itshape Ключові слова:} {знакопочережний ряд Катора,  система функціональних рівнянь, монотонна функція, неперервна ніде немонотонна функція, сингулярна функцiя, ніде недиференційовна функція, функція розподілу.}


\end{center}

\maketitle

\section{Вступ}
В сучасних наукових дослідженнях для задання функцій зі складною локальною будовою (сингулярних, неперервних ніде недиференційовних) широко використовуються різні представлення дійсних чисел. Наприклад, представлення чисел знакододатними та знакопочережними рядами, члени яких є числами, оберненими до натуральних або добутків натуральних чисел.

У даній статті розглядається застосування до побудови монотонних сингулярних та неперервних ніде немонотонних функцій розкладів дійсних чисел в нескінченні ряди, члени яких є раціональними числами. Традиційно найпростішими прикладами таких рядів вважають $s$-ве та нега-$s$-ве зображення.

В \cite{Cantor1} Георг Кантор вперше обгрунтував можливість представлення будь-якого дійсного числа $x \in[0;1]$ у вигляді розкладу в знакододатний ряд 
$$
\sum^{\infty} _{n=1} {\frac{\varepsilon_n}{d_1d_2...d_n}},
$$
де $(d_n)$ --- фіксована послідовність натуральних чисел $d_n$, більших $1$, $(A_{d_n})$ --- послідовність алфавітів $A_{d_n} \equiv \{0,1,...,d_n-1\}$, $\varepsilon_n\in A_{d_n}$.

Останній ряд в науковій літературі називають знакододатним рядом Кантора (або рядом Кантора). Очевидно, представлення дійсних чисел знакододатними рядами Кантора є узагальненням $s$-ої системи числення. Тому логічно було б припустити можливість зображення дійсних чисел у вигляді розкладу в \emph{знакопочережний ряд Кантора}
\begin{equation}
\label{eq: alternating Cantor series 1}
\sum^{\infty} _{n=1} {\frac{(-1)^n\varepsilon_n}{d_1d_2...d_n}},
\end{equation}
 який є узагальненням нега-$s$-го розкладу дійсного числа.

У даній роботі   зосереджено увагу на дослідженні основних властивостей функцій зі складною локальною будовою, аргумент яких представлений знакопочережним рядом Кантора виду
\begin{equation}
\label{eq: alternating Cantor series 2}
\sum^{\infty} _{n=1} {\frac{1+\varepsilon_n}{d_1d_2...d_n}(-1)^{n+1}},~\mbox{де}~\varepsilon_n\in A_{d_n},
\end{equation}
оскільки областю визначення досліджуваних функцій є відрізок $[0;1]$,  довільне число з якого можна представити саме у вигляді розкладу в останній ряд.

Перш ніж перейти до висвітлення основних результатів дослідження, розглянемо питаня про приналежність розкладу \eqref{eq: alternating Cantor series 1} і, як наслідок, розкладу \eqref{eq: alternating Cantor series 2} до систем числення.

\section{ Теорема про розклад числа в знакопочережний ряд Кантора}

\begin{theorem}
Кожне число $x \in [a_0-1;a_0]$, де
$$
a_0=\sum^{\infty} _{n=1} {\frac{(-1)^{n+1}}{d_1d_2...d_n}},
$$
 можна представити у вигляді розкладу в знакопочережний  ряд Кантора \eqref{eq: alternating Cantor series 1} не більш ніж двома способами.
\end{theorem}

Аналогічне твердження для $x \in [0;1]$ справедливе і для знакопочережних рядів~\eqref{eq: alternating Cantor series 2}, оскільки послідовність $(d_n)$ є фіксованою, звідки слідує, що   $a_0=const$.

\begin{definition}
Представлення довільного числа $x$ з відрізка $[a_0-1;a_0]$ (або з~$[0;1]$) у вигляді розкладу в знакопочережний ряд Кантора \eqref{eq: alternating Cantor series 1} (або \eqref{eq: alternating Cantor series 2})
називається {\it нега-$D$-представленням}, де $D\equiv(d_n)$,  ({\it або нега-$(d_n)$-представленням }) {\it числа $x$} і позначається $x=\Delta^{-D} _{\varepsilon_1\varepsilon_2...\varepsilon_n...}$ (або $x=\Delta^{-(d_n)} _{\varepsilon_1\varepsilon_2...\varepsilon_n...}$). Останні позначення називаються {\it нега-$D$-зображенням} ( або {\it нега-$(d_n)$-зображенням }) числа $x$ відповідно. 
\end{definition}

Твердження останньої теореми слідує з тверджень наступних двох лем.

\begin{lemma} 
Для будь-якого $x \in [a_0-1;a_0]$  існує послідовність $(\varepsilon_n)$ така, що число $x$ можна представити у вигляді розкладу в знакопочережний ряд Кантора \eqref{eq: alternating Cantor series 1}.
\end{lemma}
\begin{proof} Очевидно, що
$$
a_0=\max\left\{{\sum^{\infty} _{n=1} {\frac{(-1)^n\varepsilon_n}{d_1d_2...d_n}}}\right\} \equiv \Delta^{-D} _{0[d_2-1]0[d_4-1]0[d_6-1]0...},
$$
$$
a_0-1=\min\left\{{\sum^{\infty} _{n=1} {\frac{(-1)^n\varepsilon_n}{d_1d_2...d_n}}}\right\} \equiv\Delta^{-D} _{[d_1-1]0[d_3-1]0[d_5-1]0...}.
$$

 Нехай $x$ --- довільне число з $(a_0-1;a_0)$. Оскільки
$$
-\frac{\varepsilon_1}{d_1}-\sum^{\infty} _{k=2} {\frac{d_{2k-1}-1}{d_1d_2...d_{2k-1}}}< x \le -\frac{\varepsilon_1}{d_1}+\sum^{\infty} _{k=1} {\frac{d_{2k}-1}{d_1d_2...d_{2k}}}
$$
при $0 \le \varepsilon_1 \le d_1-1$, а також
$$
[a_0-1;a_0]=I_0=\bigcup^{d_1-1} _{i=0} {\left[-\frac{i}{d_1}- \sum^{\infty} _{k=2} {\frac{d_{2k-1}-1}{d_1d_2...d_{2k-1}}}; -\frac{i}{d_1}+\sum^{\infty} _{k=1} {\frac{d_{2k}-1}{d_1d_2...d_{2k}}} \right]},
$$
тому
$$
- \sum^{\infty} _{k=2} {\frac{d_{2k-1}-1}{d_1d_2...d_{2k-1}}}<x+\frac{\varepsilon_1}{d_1}\le \sum^{\infty} _{k=1} {\frac{d_{2k}-1}{d_1d_2...d_{2k}}}.
$$

Позначимо $x+\frac{\varepsilon_1}{d_1}=x_1$. Отримаємо випадки:

\begin{enumerate}
\item 
$$
x_1= \sum^{\infty} _{k=1} {\frac{d_{2k}-1}{d_1d_2...d_{2k}}}.
$$

В такому разі отримаємо, що 
$$
x=\Delta^{-D} _{\varepsilon_1[d_2-1]0[d_4-1]0...} ~\mbox{або} ~ x=\Delta^{-D} _{[\varepsilon_1-1]0[d_3-1]0[d_5-1]0...}.
$$
\item Якщо не справджується рівність, зазначена у першому випадку, тоді 
$$
x=-\frac{\varepsilon_1}{d_1}+x_1, ~\mbox{де}
$$
$$
\frac{\varepsilon_2}{d_1d_2}-\sum^{\infty} _{k=2} {\frac{d_{2k-1}-1}{d_1d_2...d_{2k-1}}}\le x_1<\frac{\varepsilon_2}{d_1d_2}+\sum^{\infty} _{k=2} {\frac{d_{2k}-1}{d_1d_2...d_{2k}}}.
$$
\end{enumerate}
Позначивши  $x_2=x_1-\frac{\varepsilon_2}{d_1d_2}$, отримаємо знову ж таки два випадки:
\begin{enumerate}
\item 
$$
x_2=\sum^{\infty} _{k=2} {\frac{d_{2k-1}-1}{d_1d_2...d_{2k-1}}}.
$$
У такому разі
$$
x=\Delta^{-D} _{\varepsilon_1\varepsilon_2[d_3-1]0[d_5-1]0...} ~\mbox{або} ~ x=\Delta^{-D} _{\varepsilon_1[\varepsilon_2-1]0[d_4-1]0[d_6-1]0...}.
$$
\item Якщо  умова першого випадку знову ж таки не виконується, отримаємо
$$
x=-\frac{\varepsilon_1}{d_1}+\frac{\varepsilon_2}{d_1d_2}+x_2, ~\mbox{де}
$$
$$
-\frac{\varepsilon_3}{d_1d_2d_3}-\sum^{\infty} _{k=3} {\frac{d_{2k-1}-1}{d_1d_2...d_{2k-1}}}<x_2 \le -\frac{\varepsilon_3}{d_1d_2d_3}+\sum^{\infty} _{k=2} {\frac{d_{2k}-1}{d_1d_2...d_{2k}}} ~\mbox{і т. д.}
$$
\end{enumerate}

За скінченну кількість кроків $m$ отримаємо подвійну строгу нерівність
$$
\frac{(-1)^{m+1}\varepsilon_{m+1}}{d_1d_2...d_{m+1}}-\sum_{k>\frac{m+2}{2}} {\frac{d_{2k-1}-1}{d_1d_2...d_{2k-1}}}<x_m<\frac{(-1)^{m+1}\varepsilon_{m+1}}{d_1d_2...d_{m+1}}+\sum_{k>\frac{m+1}{2}} {\frac{d_{2k}-1}{d_1d_2...d_{2k}}},
$$
але в залежності від парності $m$ одна з нерівностей за певних умов може перетворюватись в рівність. Тобто,  можливі випадки:
\begin{enumerate}
\item 
$$
x_{m+1}=\begin{cases}
\sum_{k>\frac{m+2}{2}} {\frac{d_{2k-1}-1}{d_1d_2...d_{2k-1}}},&\text{якщо $m$ --- непарне;}\\
\\
\sum_{k>\frac{m+1}{2}} {\frac{d_{2k}-1}{d_1d_2...d_{2k}}},&\text{якщо $m$ --- парне.}
\end{cases}
$$
В такому разі 
$$
x=\Delta^{-D} _{\varepsilon_1\varepsilon_2...\varepsilon_{m+1}[d_{m+2}-1]0[d_{m+4}-1]0...} ~\mbox{або} ~ x=\Delta^{-D} _{\varepsilon_1...\varepsilon_{m}[\varepsilon_{m+1}-1]0[d_{m+3}-1]0[d_{m+5}-1]0...}.
$$
\item У випадку, коли не існує такого $m \in \mathbb N$, щоб виконувалась хоча б одна з умов останньої системи, отримаємо
$$
x=\sum^{m+1} _{n=1}{\frac{(-1)^n\varepsilon_n}{d_1d_2...d_n}}+x_{m+1}.
$$
 Продовжуючи процес до нескінченності, отримаємо
$$
x=-\frac{\varepsilon_1}{d_1}+x_1=-\frac{\varepsilon_1}{d_1}+\frac{\varepsilon_2}{d_1d_2}+x_2=...=
$$
$$
=-\frac{\varepsilon_1}{d_1}+\frac{\varepsilon_2}{d_1d_2}-\frac{\varepsilon_3}{d_1d_2d_3}+...+\frac{(-1)^n\varepsilon_n}{d_1d_2...d_n}+x_n=....=\sum^{\infty} _{n=1} {\frac{(-1)^n \varepsilon_n}{d_1d_2...d_n}}.
$$
\end{enumerate}
\end{proof}

\begin{lemma} 
Числа $x=\Delta^{-D} _{\varepsilon_1\varepsilon_2...\varepsilon_{m-1}\varepsilon_m \varepsilon_{m+1}...}$ та $x^{'}=~\Delta^{-D} _{\varepsilon_1\varepsilon_2...\varepsilon_{m-1}\varepsilon^{'} _m\varepsilon^{'} _{m+1}...}$, де  $\varepsilon_m \ne \varepsilon^{'} _m$, співпадають тоді і тільки тоді, коли  справедливою є одна із систем:
$$
\left\{
\begin{array}{rcl}
\varepsilon_{m+2i-1}&=&d_{m+2i-1}-1,\\
\varepsilon_{m+2i}&= 0 &=\varepsilon^{'} _{m+2i-1},\\
\varepsilon^{'} _{m+2i}&=&d_{m+2i}-1,\\
\varepsilon^{'} _m & = &\varepsilon_m-1;\\
\end{array}
\right.
\mbox{або}
\left\{
\begin{array}{rcl}
\varepsilon_{m+2i}&=&d_{m+2i}-1,\\
\varepsilon_{m+2i-1}&= 0 &=\varepsilon^{'} _{m+2i},\\
\varepsilon^{'} _{m+2i-1}&=&d_{m+2i-1}-1,\\
\varepsilon^{'} _m -1& = &\varepsilon_m;\\
\end{array}
\right.
$$
для всіх $i \in \mathbb N$.
\end{lemma}
\begin{proof} {\itshape Необхідність.} Нехай $\varepsilon_m= \varepsilon^{'} _m +1 $. Тоді
$$
0=x-x^{'}=\Delta^{-D} _{\varepsilon_1\varepsilon_2...\varepsilon_{m-1}\varepsilon_m \varepsilon_{m+1}...}-\Delta^{-D} _{\varepsilon_1\varepsilon_2...\varepsilon_{m-1}\varepsilon^{'} _m\varepsilon^{'} _{m+1}...}=
$$
$$
=\frac{(-1)^m}{d_1d_2...d_m}+\frac{(-1)^{m+1}(\varepsilon_{m+1}-\varepsilon^{'} _{m+1})}{d_1d_2...d_{m+1}}+...+
$$
$$
+\frac{\varepsilon_{m+i}-\varepsilon^{'} _{m+i}}{d_1d_2...d_{m+i}}(-1)^{m+i}+...=\frac{(-1)^m}{d_1d_2...d_m}\left(1+\sum^{\infty} _{i=1} {\frac{(-1)^i (\varepsilon_{m+i}-\varepsilon^{'} _{m+i})}{d_{m+1}d_{m+2}...d_{m+i}}}\right).
$$
$$
v \equiv \sum^{\infty} _{i=1} {\frac{(-1)^i(\varepsilon_{m+i}-\varepsilon^{'} _{m+i})}{d_{m+1}d_{m+2}...d_{m+i}}} \ge - \sum^{\infty} _{i=1}{\frac{d_{m+i}-1}{d_{m+1}d_{m+2}...d_{m+i}}}=-1.
$$
Остання нерівність перетворюється в рівність лише в тому випадку, коли 
$$
\varepsilon_{m+2i}=\varepsilon^{'} _{m+2i-1}=0~\mbox{ і} ~\varepsilon_{m+2i-1}=d_{m+2i-1}-1, ~\varepsilon^{'} _{m+2i}=d_{m+2i}-1.
$$

Тобто, в даному випадку з рівності $x=x^{'}$ випливають умови першої системи. Аналогічним чином, з умови $x=x^{'}$ при $\varepsilon^{'} _m=\varepsilon_m+1$  випливають умови другої системи.

{\it Достатність} є очевидною.\end{proof}

При подальшому дослідженні використовуватимуться допоміжні поняття  нега-$(d_n)$-раціонального та нега-$(d_n)$-ірраціонального чисел.

\begin{definition}
Числа, що мають  два різних нега-$(d_n)$-зображення, а саме:
$$
\Delta^{-(d_n)} _{\varepsilon_1\varepsilon_2...\varepsilon_{n-1}\varepsilon_n[d_{n+1}-1]0[d_{n+3}-1]...}=\Delta^{-(d_n)} _{\varepsilon_1\varepsilon_2...\varepsilon_{n-1}[\varepsilon_n-1]0[d_{n+2}-1]0[d_{n+4}-1]...},
$$
  називаються {\it нега-$(d_n)$-раціональними}. Решта чисел з $[0;1]$  називаються {\it нега-$(d_n)$-ірраціональними} і мають єдине нега-$(d_n)$-зображення.
\end{definition}

\section{ Об'єкт дослідження.}
\label{3}

Нехай 
$P=||p_{i,n}||$ --- задана матриця, така, що $n=1,2,...$, $i=\overline{0,d_n-1}$, і для якої справедливою є наступна система властивостей:
$$ 
\left\{
\begin{aligned}
\label{eq: tilde Q 1}
1^{\circ}.~~~~~~~~~~~~~~~~~\forall n \in \mathbb N:  p_{i,n}\in (-1;1);\\
2^{\circ}.  ~~~~~~~~~~~~~~~~~~~\forall n \in \mathbb N: \sum^{d_n-1}_{i=0} {p_{i,n}}=1;\\
3^{\circ}. ~~~~~~~~ \forall (i_n), i_n \in  A_{d_n}: \prod^{\infty} _{n=1} {|p_{i_n,n}|}=0;\\
4^{\circ}.~~~~~~~~~\forall  i_n \in A_{d_n}\setminus\{0\}: \sum^{i_n-1} _{i=0} {p_{i,n}}>0.\\
\end{aligned}
\right.
$$ 

Нехай $x=\Delta^{-(d_n)} _{\varepsilon_1\varepsilon_2...\varepsilon_n...}$. Розглянемо функцію виду
$$ 
\tilde{F}(x)=\beta_{\varepsilon_1(x),1}+\sum^{\infty} _{n=2} {\left(\tilde{\beta}_{\varepsilon_n(x),n}\prod^{n-1} _{j=1} {\tilde{p}_{\varepsilon_j(x),j}}\right)},
$$
де
$$
\beta_{\varepsilon_{n},n}=\begin{cases}
0,&\text{якщо $\varepsilon_{n}=0$;}\\
\sum^{\varepsilon_{n}-1} _{i=0} {p_{i,n}}>0,&\text{якщо $\varepsilon_{n}\ne 0$,}
\end{cases}
$$
$$
\tilde{\beta}_{\varepsilon_n(x),n}=\begin{cases}
\beta_{\varepsilon_n(x),n},&\text{якщо $n$ --- непарне;}\\
\beta_{d_n-1-\varepsilon_n(x),n},&\text{якщо $n$ --- парне,}
\end{cases}
$$
$$
\tilde{p}_{\varepsilon_n(x),n}=\begin{cases}
p_{\varepsilon_n(x),n},&\text{якщо $n$ --- непарне;}\\
p_{d_n-1-\varepsilon_n(x),n},&\text{якщо $n$ --- парне.}
\end{cases}
$$

Дослідимо наявність інших способів для  задання досліджуваних функцій.
Для цього скористаємося взаємозв'язком знакопочережного та знакододатного канторівських представлень:
$$
\sum^{\infty} _{n=1} {\frac{1+\varepsilon_n}{d_1d_2...d_n}(-1)^{n+1}}\equiv \Delta^{-(d_n)} _{\varepsilon_1\varepsilon_2...\varepsilon_n...} \equiv \Delta^{D} _{\varepsilon_1[d_2-1-\varepsilon_2]\varepsilon_3[d_4-1-\varepsilon_4]...} \equiv \frac{\varepsilon_1}{d_1}+\frac{d_2-1-\varepsilon_2}{d_1d_2}+....
$$

З останнього випливає, що
$$
\tilde{F}(x)=F(g(x))=F\circ g,
$$
де
$$
x=\Delta^{-(d_n)} _{\varepsilon_1\varepsilon_2...\varepsilon_n...} \stackrel{g}{\rightarrow} \Delta^{D} _{\varepsilon_1[d_2-1-\varepsilon_2]...\varepsilon_{2n-1}[d_{2n}-1-\varepsilon_{2n}]...}=g(x)=y,
$$
$$
F\left(\sum^{\infty} _{n=1} {\frac{\varepsilon_n}{d_1d_2...d_n}} \right)=\beta_{\varepsilon_1,1}+\sum^{\infty} _{n=2} {\left(\beta_{\varepsilon_n,n}\prod^{n-1} _{j=1} {p_{\varepsilon_j,j}}\right)}.
$$

Корисним  при дослідженні різних способів задання функції $\tilde{F}$ є поняття оператора зсуву цифр представлення дійсного числа знакододатним рядом Кантора.

\begin{definition}
{\it Оператором зсуву цифр представлення числа $x=\Delta^D _{\varepsilon_1\varepsilon_2...\varepsilon_n...}$ знакододатним рядом Кантора} називається відображення $\hat{\varphi}$, задане наступним чином:
$$
\hat{\varphi}(x)=\hat{\varphi} \left (\sum^{\infty} _{n=1}{\frac{\varepsilon_n}{d_1d_2...d_n}}\right)=\sum^{\infty} _{n=2} {\frac{\varepsilon_n}{d_2d_3...d_n}}.
$$
\end{definition}

 Точніше, $\hat{\varphi}(x)=d_1x-\varepsilon_1(x)\equiv d_1\Delta^D _{0\varepsilon_2\varepsilon_3...}$. Варто відмітити, що
$$
\hat{\varphi}^k(x)=\sum^{\infty} _{n=k+1} {\frac{\varepsilon_n}{d_{k+1}d_{k+2}...d_n}}\equiv  d_1d_2...d_{k}\Delta^D _{\underbrace{0...0}_{k}\varepsilon_{k+1}\varepsilon_{k+2}\varepsilon_{k+3}...},
$$
\begin{equation}
\label{eq:rivnist1(2)}
\frac{i}{d_k}+\frac{\hat{\varphi}^k(x)}{d_k}=d_1d_2...d_{k-1}\Delta^D _{\underbrace{0...0}_{k-1}i\varepsilon_{k+1}\varepsilon_{k+2}\varepsilon_{k+3}...}=\hat{\varphi}^{k-1}(x).
\end{equation}

Легко показати, що досліджувана функція  $\tilde{F}$   в класі визначених та обмежених на відрізку функцій є єдиним розв'язком наступних (еквівалентних між собою в силу рівності \eqref{eq:rivnist1(2)}) нескінченних систем функціональних рівнянь. А саме:
\begin{itemize}
\item 
$$
f\left(\frac{\tilde i(x)+\hat{\varphi}^k (y)}{d_k}\right)=\tilde{\beta}_{i(x),k}+\tilde{p}_{i(x),k}\cdot f(\hat{\varphi}^k (y)), 
$$
де $k=1,2,...,$ $i \in A_{d_k}$,
$$
\tilde i(x)=\begin{cases}
i(x),&\text{якщо $k$ --- непарне;}\\
d_k-1-i(x),&\text{якщо $k$ --- парне;}
\end{cases}
$$
\item $$
f(\hat{\varphi}^k (y))=\tilde{\beta}_{\varepsilon_{k+1}(x),k+1}+\tilde{p}_{ \varepsilon_{k+1}(x),k+1}f(\hat{\varphi}^{k+1} (y)), 
$$
де \mbox{$k=0,1,2,...,$ $\varepsilon_{k+1} \in A_{d_{k+1}}$ i  $y=~\Delta^D _{\varepsilon_1[d_2-1-\varepsilon_2]\varepsilon_3...[d_{2n}-1-\varepsilon_{2n}]\varepsilon_{2n+1}...}$.}
\end{itemize}

Справді, 
$$
\tilde{F}(x)=\beta_{\varepsilon_1(x),1}+\sum^{k} _{n=2} {\left(\tilde{\beta}_{\varepsilon_n(x),n}\prod^{n-1} _{j=1} {\tilde{p}_{\varepsilon_j(x),j}}\right)}+\left(\prod^{k} _{j=1} {\tilde{p}_{\varepsilon_j(x),j}}\right)\cdot f(\hat{\varphi}^k(y)),
$$
де $y=\Delta^D _{\varepsilon_1(x)[d_2-1-\varepsilon_2(x)]\varepsilon_3(x)[d_4-1-\varepsilon_4(x)]...}$. Оскільки функція $\tilde{F}$ є визначеною та обмеженою на $[0;1]$, в силу третьої властивості матриці $P$ при граничному переході в останній рівності при $k \to \infty$ легко помітити справедливість доводжуваного твердження.

Самим важливим твердженням даного пункту статті є відповідь на питання про коректність означеної функції. Тобто, під коректністю функції розуміється  неіснування двох і більше різних значень функції для довільного аргументу з її області визначення. 

\begin{lemma} 
Функція $y=\tilde{F}(x)$ є коректно означеною в довільній точці відрізка $[0;1]$. 
\end{lemma}
\begin{proof} Для доведення твердження леми достатньо розглянути випадок, коли аргумент функції є  нега-$(d_n)$-раціональним числом. Нехай $x$ --- нега-$(d_n)$-раціональне число. 

Розглянемо  різницю 
$$
\delta=\tilde{F}(\Delta^{-(d_n)} _{\varepsilon_1...\varepsilon_{n-1}\varepsilon_n[d_{n+1}-1]0[d_{n+3}]0...})-\tilde{F}(\Delta^{-(d_n)} _{\varepsilon_1...\varepsilon_{n-1}[\varepsilon_n-1]0[d_{n+2}-1]0[d_{n+4}-1]0...})=
$$
$$
=\left(\prod^{n-1} _{j=1} {\tilde{p}_{\varepsilon_j,j}}\right)\cdot [(\tilde{\beta}_{\varepsilon_n,n}+\tilde{\beta}_{d_{n+1}-1,n+1}\tilde{p}_{\varepsilon_n,n}+\tilde{\beta}_{0,n+2}\tilde{p}_{\varepsilon_n,n}\tilde{p}_{d_{n+1}-1,n+1}+
$$
$$
+\tilde{\beta}_{d_{n+3}-1,n+3}\tilde{p}_{\varepsilon_n,n}\tilde{p}_{d_{n+1}-1,n+1}\tilde{p}_{0,n+2}+...)-(\tilde{\beta}_{\varepsilon_n-1,n}+\tilde{\beta}_{0,n+1}\tilde{p}_{\varepsilon_n-1,n}+
$$
$$
+\tilde{\beta}_{d_{n+2}-1,n+2}\tilde{p}_{\varepsilon_n-1,n}\tilde{p}_{0,n+1}+\tilde{\beta}_{0,n+3}\tilde{p}_{\varepsilon_n-1,n}\tilde{p}_{0,n+1}\tilde{p}_{d_{n+2}-1,n+2}+...)].
$$

Для  парного $n$ 
$$
\delta=\left(\prod^{n-1} _{j=1} {\tilde{p}_{\varepsilon_j,j}}\right)[\beta_{ d_n-1-\varepsilon_n,n}+\beta_{d_{n+1}-1,n+1}p_{d_n-1-\varepsilon_n,n}+
$$
$$
+p_{ d_n-1-\varepsilon_n,n}\sum^{\infty} _{k=2} {\left({\beta_{d_{n+k}-1,n+k}}\prod^{k-1} _{j=1} {p_{d_{n+j}-1,n+j}}\right)}]-\left(\prod^{n-1} _{j=1} {\tilde{p}_{\varepsilon_j,j}}\right)\times
$$
$$
\times\left[\beta_{d_n-\varepsilon_n,n}+\beta_{0,n+1}p_{ d_n-\varepsilon_n,n}+p_{d_n-\varepsilon_n,n}\sum^{\infty} _{k=2} {\left({\beta_{0,n+k}}\prod^{k-1} _{j=1} {p_{ 0,n+j}}\right)}\right]=
$$
$$
=\left(\prod^{n-1} _{j=1} {\tilde{p}_{\varepsilon_j,j}}\right)\cdot (-{p}_{d_n-\varepsilon_n-1,n}+(1-p_{d_{n+1}-1,n+1}){p}_{d_n-\varepsilon_n-1,n}+
$$
$$
+{p}_{d_n-\varepsilon_n-1,n}\sum^{\infty} _{k=2} {\left[{(1-p_{d_{n+k}-1,n+k})}\prod^{k-1} _{j=1} {p_{ d_{n+j}-1,n+j}}\right]})=0.
$$

У випадку непарного $n$ отримаємо
$$
\delta=({p}_{\varepsilon_n-1,n}-(1-{p}_{d_{n+1}-1,n+1}){p}_{\varepsilon_n-1,n}-(1-{p}_{d_{n+2}-1,n+2}){p}_{\varepsilon_n-1,n}{p}_{d_{n+1}-1,n+1}-...)\prod^{n-1} _{j=1} {\tilde{p}_{\varepsilon_j,j}}=0.
$$
\end{proof}

Перейдемо до розгляду властивостей функції $\tilde{F}(x)$.

\section{ Неперервність та монотонність}

\begin{theorem} Функція $\tilde{F}$ є : 
\begin{itemize}
\item неперервною;
\item монотонно неспадною за умови невід'ємності елементів матриці $P$, зокрема, строго зростаючою за умови,  коли всі елементи $p_{i,n}$ матриці  $P$ є додатними;
\end{itemize}
\end{theorem}
\begin{proof} {\itshape Неперервність.}
Нехай маємо довільне число $ x_0=\Delta^{-(d_n)} _{\varepsilon_1(x_0)\varepsilon_2(x_0)...\varepsilon_{n_0}(x_0)\varepsilon_{n_0+1}(x_0)...}$ з відрізка $[0;1]$. Нехай $x=\Delta^{-(d_n)} _{\varepsilon_1(x)\varepsilon_2(x)...\varepsilon_{n_0}(x)\varepsilon_{n_0+1}(x)...}$ --- таке число, для якого справедливими є умови, що $\varepsilon_j(x)=\varepsilon_j(x_0)$ при $j=\overline{1,n_0-1}$ та  $\varepsilon_{n_0}(x)\ne \varepsilon_{n_0}(x_0)$.  Розглянемо різницю 
$$
\tilde{F}(x)-\tilde{F}(x_0)=\left(\prod^{n_0-1} _{j=1} {\tilde{p}_{\varepsilon_j(x_0),j}}\right)\left(\tilde{F}(\hat{\varphi}^{n_0-1}(x))-\tilde{F}(\hat{\varphi}^{n_0-1}(x_0))\right).
$$
Таким чином,

$$
|\tilde{F}(x)-\tilde{F}(x_0)|\le \left(\prod^{n_0-1} _{j=1} {|\tilde{p}_{\varepsilon_j(x_0),j}}|\right)\le \left(\max_{j=\overline{1,n_0-1}} {|\tilde{p}_{\varepsilon_j(x_0),j}|} \right)^{n_0-1} \to 0 ~~~(n_0 \to \infty),
$$
що еквівалентно умові $\lim_{x \to x_0} {\tilde{F}(x)}=\tilde{F}(x_0)$.

Справді, для нега-$(d_n)$-ірраціонального числа $x_0$ умови $x \to x_0$ та  $n_0\to \infty$ є еквівалентними  і тому не виникає сумнівів щодо  неперервності функції $\tilde{F}$.

Нехай $x_0$ --- нега-$(d_n)$-раціональне число. У такому разі неперервність функції $\tilde{F}$ в нега-$(d_n)$-раціональній точці $x_0$ можна довести, використовуючи поняття односторонніх границь з врахуванням випадків парного та непарного $n_0$.

{\it{Монотонність.}}
Нехай елементи $p_{i,n}$ матриці $P$ є невід'ємними. Очевидно, що
$$
\tilde{F}(0)=\tilde{F}(\Delta^{-(d_n)} _{0[d_2-1]0[d_4-1]...})=\beta_{0,1}+\sum^{\infty} _{n=2} {\left({\beta}_{0,n}\prod^{n-1} _{j=1} {{p}_{0,j}}\right)}=\min_{x \in [0;1]} {\tilde{F}(x)}=0,
$$
$$
\tilde{F}(1)=\tilde{F}(\Delta^{-(d_n)} _{[d_1-1]0[d_3-1]0...})=\beta_{d_1-1,1}+\sum^{\infty} _{n=2} {\left({\beta}_{d_n-1,n}\prod^{n-1} _{j=1} {{p}_{d_j-1,j}}\right)}=
$$
$$
=\max_{x \in [0;1]} {\tilde{F}(x)}=1.
$$

 Нехай  $x_1=\Delta^{-(d_n)} _{\varepsilon_1(x_1)\varepsilon_2(x_1)...\varepsilon_{n}(x_1)...}$ та $x_2=~\Delta^{-(d_n)} _{\varepsilon_1(x_2)\varepsilon_2(x_2)...\varepsilon_{n}(x_2)...}$~--- деякі числа  ($x_1<~x_2$).  Очевидно, існує такий номер $n_0$, що $\varepsilon_j(x_1)=~\varepsilon_j(x_2)$ для всіх $j=~\overline{1,n_0-1}$ та $\varepsilon_{n_0}(x_1)<~\varepsilon_{n_0}(x_2)$ при непарному $n_0$ або $\varepsilon_{n_0}(x_1)>\varepsilon_{n_0}(x_2)$ у випадку парного $n_0$.

Отже,
$$
\tilde{F}(x_2)-\tilde{F}(x_1)=\left(\prod^{n_0-1} _{j=1} {\tilde{p}_{\varepsilon_{j}(x_2),j}}\right)\cdot(\tilde{\beta}_{\varepsilon_{n_0}(x_2),n_0}-\tilde{\beta}_{\varepsilon_{n_0}(x_1),n_0}+
$$
$$
+\sum^{\infty} _{m=1} {\left(\tilde{\beta}_{\varepsilon_{n_0+m}(x_2),n_0+m}\prod^{m-1} _{j=0} {\tilde{p}_{\varepsilon_{n_0+j}(x_2),n_0+j}}\right)}
-\sum^{\infty} _{m=1} {\left(\tilde{\beta}_{\varepsilon_{n_0+m}(x_1),n_0+m}\prod^{m-1} _{j=0} {\tilde{p}_{\varepsilon_{n_0+j}(x_1),n_0+j}}\right)}).
$$
Оскільки
$$
\kappa=\sum^{\infty} _{m=1} {\left(\tilde{\beta}_{\varepsilon_{n_0+m}(x_2),n_0+m}\prod^{m-1} _{j=0} {\tilde{p}_{\varepsilon_{n_0+j}(x_2),n_0+j}}\right)}
-\sum^{\infty} _{m=1} {\left(\tilde{\beta}_{\varepsilon_{n_0+m}(x_1),n_0+m}\prod^{m-1} _{j=0} {\tilde{p}_{\varepsilon_{n_0+j}(x_1),n_0+j}}\right)}\ge 
$$
$$
\ge -\sum^{\infty} _{m=1} {\left(\tilde{\beta}_{\varepsilon_{n_0+m}(x_1),n_0+m}\prod^{m-1} _{j=0} {\tilde{p}_{\varepsilon_{n_0+j}(x_1),n_0+j}}\right)},
$$
де для непарного $n_0$
$$
\kappa \ge -p_{\varepsilon_{n_0}(x_1),n_0}(1-p_{d_{n_0+1}-1,n_0+1}+ 
\sum^{\infty} _{m=2} {\left[(1-p_{d_{n_0+m}-1,n_0+m})\prod^{m-1} _{j=1} {p_{d_{n_0+j}-1,n_0+j}}\right]})=-p_{\varepsilon_{n_0}(x_1),n_0}
$$
та для парного  $n_0$ по аналогії отримаємо
$$
\kappa \ge  -p_{d_{n_0}-1-\varepsilon_{n_0}(x_1),n_0}\cdot \left (\max_{x \in [0,1]} {\tilde{F}(\hat{\varphi}^{n_0}(x_1))}\right) = -p_{d_{n_0}-1-\varepsilon_{n_0}(x_1),n_0}.
$$
Як наслідок, у випадку непарного $n_0$
$$
\tilde{F}(x_2)-\tilde{F}(x_1) =\left(\prod^{n_0-1} _{j=1} {\tilde{p}_{\varepsilon_{j}(x_2),j}}\right)\cdot (\tilde{\beta}_{\varepsilon_{n_0}(x_2),n_0}-\tilde{\beta}_{\varepsilon_{n_0}(x_1),n_0}+\kappa)\ge 
$$
$$
\ge \left(\prod^{n_0-1} _{j=1} {\tilde{p}_{\varepsilon_{j}(x_2),j}}\right)\cdot (p_{\varepsilon_{n_0}(x_1),n_0}+p_{\varepsilon_{n_0}(x_1)+1,n_0}+...+p_{\varepsilon_{n_0}(x_2)-1,n_0}- p_{\varepsilon_{n_0}(x_1),n_0})\ge 0
$$
та у випадку парного $n_0$
$$
\tilde{F}(x_2)-\tilde{F}(x_1) =\left(\prod^{n_0-1} _{j=1} {\tilde{p}_{\varepsilon_{j}(x_2),j}}\right)\cdot (\tilde{\beta}_{\varepsilon_{n_0}(x_2),n_0}-\tilde{\beta}_{\varepsilon_{n_0}(x_1),n_0}+\kappa)= 
$$
$$
=\left(\prod^{n_0-1} _{i=1} {\tilde{p}_{i,\varepsilon_{i}(x_2)}}\right)\cdot (p_{d_{n_0}-1-\varepsilon_{n_0}(x_1),n_0}+p_{d_{n_0}-\varepsilon_{n_0}(x_1),n_0}+...+
$$
$$
+p_{d_{n_0}-2-\varepsilon_{n_0}(x_2),n_0}-p_{d_{n_0}-1-\varepsilon_{n_0}(x_1),n_0})\ge 0.
$$

Цілком очевидно, що коли коли всі елементи $p_{i,n}$ матриці  $P$ є додатними, справедливою буде умова $\tilde{F}(x_2)-\tilde{F}(x_1) >0$. 
\end{proof}

Нехай елементи $p_{i,n}$ матриці $P$ є невід'ємними. 

Нехай $\eta$ --- випадкова величина, представлена у вигляді наступного канторівського розкладу 
$$
\eta= \frac{\xi_1}{d_1}+\frac{\xi_2}{d_1d_2}+\frac{\xi_3}{d_1d_2d_3}+...+\frac{\xi_{k}}{d_1d_2...d_{k}}+...\equiv \Delta^{D} _{\xi_1\xi_2...\xi_{k}...},
$$
де
$$
\xi_k=\begin{cases}
\varepsilon_k,&\text{якщо $k$ --- непарне;}\\
d_k-1-\varepsilon_k,&\text{якщо $k$ --- парне.}
\end{cases}
$$
та  цифри $\xi_k$ $(k=1,2,3,...)$ є випадковими і набувають значень $0,1,...,d_k-1$ з ймовірностями ${p}_{0,k}, {p}_{1,k}, ..., {p}_{d_k-1,k}$. Тобто, $\xi_k$ --- незалежні  та $P\{\xi_k=i_k\}={p}_{i_k,k}$, $i_k \in A_{d_k}$.

В силу  означення функції розподілу та  рівностей
$$
\{\eta<x\}=\{\xi_1<\varepsilon_1(x)\}\cup\{\xi_1=\varepsilon_1(x),\xi_2<d_2-1-\varepsilon_2(x)\}\cup...\cup
$$
$$
\cup\{\xi_1=\varepsilon_1(x),\xi_2=d_2-1-\varepsilon_2(x),...,\xi_{2k-1}<\varepsilon_{2k-1}(x)\}\cup
$$
$$
\cup\{\xi_1=\varepsilon_1(x),\xi_2=d_2-1-\varepsilon_2(x),...,\xi_{2k-1}=\varepsilon_{2k-1}(x),\xi_{2k}<d_{2k}-1-\varepsilon_{2k}(x)\}\cup...,
$$
$$
P\{\xi_1=\varepsilon_1(x),\xi_2=d_2-1-\varepsilon_2(x),...,\xi_{2k-1}<\varepsilon_{2k-1}(x)\}=
$$
$$
=\beta_{\varepsilon_{2k-1}(x),2k-1}\prod^{2k-2} _{j=1} {\tilde{p}_{\varepsilon_{j}(x),j}}
$$
i
$$
P\{\xi_1=\varepsilon_1(x),\xi_2=d_2-1-\varepsilon_2(x),...,\xi_{2k}<d_{2k}-1-\varepsilon_{2k}(x)\}=
$$
$$
=\beta_{d_{2k}-1-\varepsilon_{2k}(x),2k}\prod^{2k-1} _{j=1} {\tilde{p}_{\varepsilon_{j}(x),j}},
$$
наслідком останньої теореми є наступна лема.

\begin{lemma}
Функція розподілу $\tilde{F}_{\eta}$ випадкової величини $\eta$ має вигляд
$$
\tilde{F}_{\eta}(x)=\begin{cases}
0,&\text{ $x< 0$;}\\
\beta_{\varepsilon_1(x),1}+\sum^{\infty} _{k=2} {\left[\tilde{\beta}_{\varepsilon_k(x),k} \prod^{k-1} _{j=1} {\tilde{p}_{\varepsilon_j(x),j}}\right]},&\text{ $0 \le x<1$;}\\
1,&\text{ $x\ge 1$,}
\end{cases}
$$
де $\tilde{p}_{\varepsilon_{j(x),j}} \ge 0$.
\end{lemma}

\section{ Інтегральні властивості.}

\begin{theorem}
Функція $y=\tilde{F}(x)$ є інтегровною за Лебегом, причому
$$
\int^1 _0 {\tilde{F}(x)dx}=\sum^{\infty} _{n=1} {\frac{\tilde{\beta}_{0,n}+\tilde{\beta}_{1,n}+\tilde{\beta}_{2,n}+...+\tilde{\beta}_{d_n-1,n}}{d_1d_2...d_n}}.
$$
\end{theorem}
\begin{proof}
Позначимо $y=g(x)$  (функцію $g$ було означено в пункті \ref{3}). Використовуючи означення функції $\tilde F$ ( в тому числі і властивості, що слідують з різних способів задання досліджуваної функції)  та  властивості інтеграла Лебега, отримаємо
$$
\int^1 _0 {\tilde{F}(x)dx}=\int^{\frac{1}{d_1}} _0 {{F}(y)dy}+\int^{\frac{2}{d_1}} _{\frac{1}{d_1}} {{F}(y)dy}+...+\int^{1} _{\frac{d_1-1}{d_1}} {{F}(y)dy}=
$$
$$
=\int^{\frac{1}{d_1}} _0 {p_{0,1}{F}(\hat{\varphi}(y))dy}+\int^{\frac{2}{d_1}} _{\frac{1}{d_1}} {\left[p_{0,1}+p_{1,1}{F}(\hat{\varphi}(y))\right]dy}+
$$
$$
+\int^{\frac{3}{d_1}} _{\frac{2}{d_1}} {\left[\beta_{2,1}+p_{2,1}{F}(\hat{\varphi}(y))\right]dy}+...+\int^{1} _{\frac{d_1-1}{d_1}} {\left[\beta_{d_1-1,1}+p_{d_1-1,1}{F}(\hat{\varphi}(y))\right]dy}.
$$

Оскільки $y=\frac{\varepsilon_1}{d_1}+\frac{1}{d_1}\hat{\varphi}(y)$ і, як наслідок $dy=\frac{1}{d_1}d(\hat{\varphi}(y))$, то
$$
\int^1 _0 {\tilde{F}(x)dx}=\frac{p_{0,1}}{d_1}\int^{1} _{0} {{F}(\hat{\varphi}(y))d(\hat{\varphi}(y))}+\beta_{1,1}y|^{\frac{2}{d_1}} _{\frac{1}{d_1}}+\frac{p_{1,1}}{d_1}\int^{1} _{0} {{F}(\hat{\varphi}(y))d(\hat{\varphi}(y))}+\beta_{2,1}y|^{\frac{3}{d_1}} _{\frac{2}{d_1}}+
$$
$$
+\frac{p_{2,1}}{d_1}\int^{1} _{0} {{F}(\hat{\varphi}(y))d(\hat{\varphi}(y))}+...+
\beta_{d_1-1,1}y|^{1} _{\frac{d_1-1}{d_1}}+\frac{p_{d_1-1,1}}{d_1}\int^{1} _{0} {{F}(\hat{\varphi}(y))d(\hat{\varphi}(y))}=
$$
$$
=\frac{\beta_{1,1}+\beta_{2,1}+...+\beta_{d_1-1,1}}{d_1}+\frac{1}{d_1}\int^{1} _{0} {{F}(\hat{\varphi}(y))d(\hat{\varphi}(y))}.
$$
Аналогічно, враховуючи взаємозв'язок D-зображення та нега-$(d_n)$-зображення, на другому кроці отримаємо
$$
\int^{1} _{0} {{F}(\hat{\varphi}(y))d(\hat{\varphi}(y))}=\int^{1} _{\frac{d_2-1}{d_2}} {p_{0,2}{F}(\hat{\varphi}^2(y))d(\hat{\varphi}(y))}+
$$
$$
+\int^{\frac{d_2-1}{d_2}} _{\frac{d_2-2}{d_2}} {\left[\beta_{1,2}+p_{1,2}{F}(\hat{\varphi}^2(y))\right]d(\hat{\varphi}(y))}+...+
\int^{\frac{1}{d_2}} _{0} {\left[\beta_{d_2-1,2}+p_{d_2-1,2}{F}(\hat{\varphi}^2(y))\right]d(\hat{\varphi}(y))}.
$$
Оскільки $\hat{\varphi}(y)=\frac{d_2-1-\varepsilon_2}{d_2}+\frac{1}{d_2}\hat{\varphi}^2(y)$ та $d(\hat{\varphi}(y))  = \frac{1}{d_2}d(\hat{\varphi}^2(y))$, отримаємо
$$
\int^{1} _{0} {{F}(\hat{\varphi}(y))d(\hat{\varphi}(y))}=\frac{p_{0,2}}{d_2}\int^{1} _{0} {{F}(\hat{\varphi}^2(y))d(\hat{\varphi}^2(y))}+\beta_{1,2}y|^{\frac{d_2-1}{d_2}} _{\frac{d_2-2}{d_2}}+
$$
$$
+\frac{p_{1,2}}{d_2}\int^{1} _{0} {{F}(\hat{\varphi}^2(y))d(\hat{\varphi}^2(y))}+...+
\beta_{d_2-1,2}y|^{\frac{1}{d_2}} _{0}+\frac{p_{d_2-1,2}}{d_2}\int^{1} _{0} {{F}(\hat{\varphi}^2(y))d(\hat{\varphi}^2(y))}=
$$
$$
=\frac{\beta_{1,2}+\beta_{2,2}+...+\beta_{d_2-1,2}}{d_2}+\frac{1}{d_2}\int^{1} _{0} {{F}(\hat{\varphi}^2(y))d(\hat{\varphi}^2(y))}.
$$

Отже, 
$$
\int^1 _0 {\tilde{F}(x)dx}=\frac{\beta_{1,1}+\beta_{2,1}+...+\beta_{d_1-1,1}}{d_1}+
$$
$$
+\frac{\beta_{1,2}+\beta_{2,2}+...+\beta_{d_2-1,2}}{d_1d_2}+\frac{1}{d_1d_2}\int^{1} _{0} {{F}(\hat{\varphi}^2(y))d(\hat{\varphi}^2(y))}.
$$

По аналогії, на кроці $n$ отримаємо
$$
\int^1 _0 {\tilde{F}(x)dx}=\sum^{n} _{j=1} {\frac{\tilde{\beta}_{0,j}+\tilde{\beta}_{1,j}+\tilde{\beta}_{2,j}+...+\tilde{\beta}_{d_j-1,j}}{d_1d_2...d_j}}+
$$
$$
+\frac{1}{d_1d_2...d_n}\int^{1} _{0} {{F}(\hat{\varphi}^n(y))d(\hat{\varphi}^n(y))}.
$$

Продовжуючи процес до нескінченності, отримаємо
$$
\int^1 _0 {\tilde{F}(x)dx}=\sum^{\infty} _{n=1} {\frac{\tilde{\beta}_{0,n}+\tilde{\beta}_{1,n}+\tilde{\beta}_{2,n}+...+\tilde{\beta}_{d_n-1,n}}{d_1d_2...d_n}}.
$$
\end{proof}

\section{ Диференціальні властивості у випадку  невід'ємності елементів  матриці $P$.}
Нехай елементи $p_{i,n}$ матриці $P$ є невід'ємними. 

\begin{lemma}
\label{lemma5}
Справедливими є наступні рівності:
\begin{enumerate}
\item \label{1-lemma5}
$$
\mu_{\tilde{F}}\left(\Delta^{-(d_n)} _{c_1c_2...c_n}\right)=\prod^{n} _{j=1} {\tilde{p}_{c_j,j}}\ge 0.
$$
\item  Нехай $x_0=\Delta^{-(d_n)} _{\varepsilon_1\varepsilon_2...\varepsilon_n...}$ --- нега-$(d_n)$-ірраціональна точка,  тоді
$$
{\tilde{F}}^{'}(x_0)=\lim_{n \to \infty}{\left(\prod^{n} _{j=1} {d_j\tilde{p}_{\varepsilon_j,j}}\right)}.
$$
\end{enumerate}
\end{lemma}
\begin{proof}\begin{enumerate}\item
Обчислимо приріст $\mu_{\tilde{F}}$  функції $\tilde{F}$  на циліндрах  $\Delta^{-(d_n)} _{c_1c_2...c_n}$. Тобто, на відрізках
$$
\left[\Delta^{-(d_n)} _{c_1c_2...c_{2n-1}[d_{2n}-1]0[d_{2n+2}-1]0[d_{2n+4}-1]...};\Delta^{-(d_n)} _{c_1c_2...c_{2n-1}0[d_{2n+1}-1]0[d_{2n+3}-1]...}\right],
$$
$$
\left[\Delta^{-(d_n)} _{c_1c_2...c_{2n}0[d_{2n+2}-1]0[d_{2n+4}-1]...};\Delta^{-(d_n)} _{c_1c_2...c_{2n}[d_{2n+1}-1]0[d_{2n+3}-1]0[d_{2n+5}-1]...}\right].
$$
$$
\mu_{\tilde{F}}\left(\Delta^{-(d_n)} _{c_1c_2...c_{2n-1}}\right)=\tilde{F}\left(\Delta^{-(d_n)} _{c_1c_2...c_{2n-1}0[d_{2n+1}-1]0[d_{2n+3}-1]...}\right)-
$$
$$
-\tilde{F}\left(\Delta^{-(d_n)} _{c_1c_2...c_{2n-1}[d_{2n}-1]0[d_{2n+2}-1]0[d_{2n+4}-1]...}\right)=
$$
$$
=\left(\prod^{2n-1} _{j=1} {\tilde{p}_{c_j,j}}\right)(\beta_{d_{2n}-1,2n}+\beta_{d_{2n+1}-1,2n+1}p_{d_{2n}-1,2n}
+\beta_{d_{2n+2}-1,2n+2}p_{d_{2n}-1,2n}p_{d_{2n+1}-1,2n+1}+...)=
$$
$$
=\left(\prod^{2n-1} _{j=1} {\tilde{p}_{c_j,j}}\right)(1-p_{d_{2n}-1,2n}+(1-p_{d_{2n+1}-1,2n+1})p_{d_{2n}-1,2n}+
$$
$$
+(1-p_{d_{2n+2}-1,2n+2})p_{d_{2n}-1,2n}p_{d_{2n+1}-1,2n+1}+...)=\left(\prod^{2n-1} _{j=1} {\tilde{p}_{c_j,j}}\right).
$$
Аналогічно
$$
\mu_{\tilde{F}}\left(\Delta^{-(d_n)} _{c_1c_2...c_{2n}}\right)=\tilde{F}\left(\Delta^{-(d_n)} _{c_1c_2...c_{2n}[d_{2n+1}-1]0[d_{2n+3}-1]0[d_{2n+5}-1]...}\right)-
$$
$$
-\tilde{F}\left(\Delta^{-(d_n)} _{c_1c_2...c_{2n}0[d_{2n+2}-1]0[d_{2n+4}-1]...}\right)=
$$
$$
=\left(\prod^{2n} _{j=1} {\tilde{p}_{c_j,j}}\right)(\beta_{d_{2n+1}-1,2n+1}+\beta_{d_{2n+2}-1,2n+2}p_{d_{2n+1}-1,2n+1}+
$$
$$
+\beta_{d_{2n+3}-1,2n+3}p_{d_{2n+1}-1,2n+1}p_{d_{2n+2}-1,2n+2}+...)=\left(\prod^{2n} _{j=1} {\tilde{p}_{c_j,j}}\right).
$$

Таким чином,
$$
\mu_{\tilde{F}}\left(\Delta^{-(d_n)} _{c_1c_2...c_n}\right)=\left(\prod^{n} _{j=1} {\tilde{p}_{c_j,j}}\right)\ge 0.
$$

\item Знайдемо похідну функції $\tilde{F}$ в нега-$(d_n)$-ірраціональній точці $x_0=~\Delta^{-(d_n)} _{\varepsilon_1\varepsilon_2...\varepsilon_n...}$. Оскільки
$$
x_0=\Delta^{-(d_n)} _{\varepsilon_1\varepsilon_2...\varepsilon_n...}=\bigcap^{\infty} _{n=1} {\Delta^{-(d_n)} _{\varepsilon_1\varepsilon_2...\varepsilon_n}},
$$
$$
\tilde{F}^{'}(x_0)=\lim_{n \to \infty} {\frac{\mu_{\tilde{F}}\left(\Delta^{-(d_n)} _{\varepsilon_1\varepsilon_2...\varepsilon_n}\right)}{|\Delta^{-(d_n)} _{\varepsilon_1\varepsilon_2...\varepsilon_n}|}}=\lim_{n \to \infty} {\frac{\prod^{n} _{j=1} {\tilde{p}_{\varepsilon_j,j}}}{\frac{1}{d_1d_2...d_n}}}=
$$
$$
=\lim_{n \to \infty} {\left(\prod^{n} _{j=1} {d_j\tilde{p}_{\varepsilon_j,j}}\right)}=\prod^{\infty} _{j=1} {\left({d_j\tilde{p}_{\varepsilon_j,j}}\right)}.~~~\square
$$
\end{enumerate}

Оскільки досліджувана функція  є неперервною та монотонною, то вона (згідно теореми Лебега) має скінченну похідну майже скрізь в розумінні міри Лебега. Проте, у випадку, коли $a_n=d_n\tilde{p}_{\varepsilon_n,n}>~1$ для всіх натуральних чисел $n$ за винятком, можливо, скінченної кількості, отримаємо $\tilde{F}^{'}(x_0)=\infty$. Тому:

\begin{itemize}
\item у випадку, коли для скінченної  множини значень  $n$ справджується $a_n \ge 1$, отримаємо $\tilde{F}^{'}(x_0)=0$;
\item  у випадку, коли   $a_n=1$ для всіх $n \in \mathbb N$ (що є справедливим лише для функції  $\tilde{F}(x)=x$), отримаємо $\tilde{F}^{'}(x_0)=1$;
\item у випадку, коли лише для скінченної кількості номерів справджується умова $p_{\varepsilon_n,n}\ne \frac{1}{d_n}$, отримаємо $0\le \tilde{F}^{'}(x_0)<\infty$.
\end{itemize}
\end{proof}

\section{ Самоафінність.}

\begin{theorem}
Якщо елементи $p_{i,n}$ матриці $P$ є додатними, то графік $\Gamma_{\tilde{F}}$ функції $\tilde{F}$ в просторі $\mathbb R^2$ є множиною виду
$$ 
\Gamma_{\tilde{F}}=\bigcup_{x \in [0;1]}{(x; ... \circ \psi_{\varepsilon_n,n}\circ ...\circ \psi_{\varepsilon_2,2} \circ \psi_{\varepsilon_1,1}(x))},
$$ 
де $x=\Delta^{D} _{\varepsilon_1[d_2-1-\varepsilon_2]\varepsilon_3[d_4-1-\varepsilon_4]...}$,
$$
\psi_{i_n,n}:
\left\{
\begin{aligned}
x^{'}&=\frac{1}{d_n}x+\frac{\omega_{i_n,n}}{d_n};\\
y^{'}& = \tilde{\beta}_{i_n,n}+\tilde{p}_{i_n,n}y,\\
\end{aligned}
\right.
$$
$$
\omega_{i_n,n}=\begin{cases}
i_n,&\text{якщо $n$ --- непарне;}\\
d_n-1-i_n,&\text{якщо $n$ --- парне,}
\end{cases}
$$
$i_n \in A_{d_n}$.
\end{theorem}
\begin{proof}
Оскільки умови
$$
f(x)=\beta_{i,1}+p_{i,1}f(\hat{\varphi}(x)),
$$
$$
f\left(\frac{i+x}{d_1}\right)=\beta_{i,1}+p_{i,1}f(x)
$$
є еквівалентними для $x=\Delta^{D} _{\varepsilon_1[d_2-1-\varepsilon_2]\varepsilon_3[d_4-1-\varepsilon_4]...}$, тому очевидно, що
$$
\psi_{i_1,1}:
\left\{
\begin{aligned}
x^{'}&=\frac{1}{d_1}x+\frac{i_1}{d_1};\\
y^{'}& = {\beta}_{i_1,1}+{p}_{i_1,1}y.\\
\end{aligned}
\right.
$$

Перейдемо до афінних перетворень $\psi_{i,2}$, $i=\overline{0,d_2-1}$. Оскільки рівності
$$
f(\hat{\varphi}(x))=\beta_{d_2-1-i,2}+p_{d_2-1-i,2}f(\hat{\varphi}^2(x)),
$$
$$
f\left(\frac{d_2-1-i+\hat{\varphi}(x)}{d_2}\right)=\beta_{d_2-1-i,2}+p_{d_2-1-i,2}f(\hat{\varphi}(x))
$$
є еквівалентними, тоді
$$
\psi_{i_2,2}:
\left\{
\begin{aligned}
x^{'}&=\frac{1}{d_2}x+\frac{d_2-1-i_2}{d_2};\\
y^{'}& = {\beta}_{d_2-1-i_2,2}+{p}_{d_2-1-i_2,2}y.\\
\end{aligned}
\right.
$$
По індукції отримаємо:
$$
\psi_{i_n,n}:
\left\{
\begin{aligned}
x^{'}&=\frac{1}{d_n}x+\frac{\omega_{i_n,n}}{d_n};\\
y^{'}& = \tilde{\beta}_{i_n,n}+\tilde{p}_{i_n,n}y.\\
\end{aligned}
\right.
$$

Отже, 
$$
\bigcup_{x \in [0;1]}{(x; ... \circ \psi_{\varepsilon_n,n}\circ ...\circ \psi_{\varepsilon_2,2} \circ \psi_{\varepsilon_1,1}(x))}\equiv G \subset  \Gamma_{\tilde{F}}.
$$

Нехай $T(x_0,\tilde F(x_0))\in \Gamma_{\tilde{F}}$. Розглянемо точку $x_n=\hat{\varphi}^n(x_0)$, де $x_0=~\Delta^{D} _{\varepsilon_1[d_2-1-\varepsilon_2]\varepsilon_3[d_4-1-\varepsilon_4]...}$~--- деяка фіксована точка з $[0;1]$.

В силу того, що для будь-якого $n \in \mathbb N$  $\varepsilon_n$ i $d_n-1-\varepsilon_n$ належать множині $A_{d_n}$, 
$$
f\left(\hat{\varphi}^{k}(x_0)\right)=\tilde{\beta}_{\varepsilon_{k+1},k+1}+\tilde{p}_{\varepsilon_{k+1},k+1}f\left(\hat{\varphi}^{k+1}(x_0)\right),~k=0,1,...
$$
та з того, що $\overline{T}\left(\hat{\varphi}^{k}(x_0);\tilde F\left(\hat{\varphi}^{k}(x_0)\right)\right)\in \Gamma_{\tilde{F}}$ випливає 
$$
\psi_{i_k,k}\circ ...\circ \psi_{i_2,2} \circ \psi_{i_1,1}\left(\overline{T}\right)=T_0(x_0;\tilde F(x_0))\in \Gamma_{\tilde{F}}, ~~~i_k \in A_{d_k},~~~k\to \infty.
$$
Звідси й слідує, що $\Gamma_{\tilde{F}}\subset G$. Отже,
$$
\Gamma_{\tilde{F}}=\bigcup_{x \in [0;1]}{(x; ... \circ \psi_{\varepsilon_n,n}\circ ...\circ \psi_{\varepsilon_2,2} \circ \psi_{\varepsilon_1,1}(x))}.
$$
\end{proof}

\section{ Недиференційовні функції} 

Розглянемо випадок, коли елементи матриці $P=~||p_{i,n}||$ можуть бути як невід'ємними, так і від'ємними числами. Тобто, нехай $p_{i,n} \in (-1;1)$ для всіх $n \in \mathbb N$, $i=\overline{0,d_n-1}$.

В такому разі з пункту \ref{1-lemma5} леми \ref{lemma5} слідує, що якщо для кожного $n \in \mathbb N $ числа  $p_{i,n}$, $i=\overline{0,d_n-1}$, є як невід'ємними, так і від'ємними числами, то функція $\tilde F$ не має жодного проміжку монотонності.

\begin{theorem}
Нехай  $p_{\varepsilon_n,n}\cdot p_{\varepsilon_n-1,n}<0$  для всіх $n \in \mathbb N$, $\varepsilon_n \in A_{d_n} \setminus \{0\}$ та
$$
\lim_{n \to \infty} {\prod^{n} _{k=1} {d_k p_{0,k}}}\ne  0, \lim_{n \to \infty} {\prod^{n} _{k=1} {d_k p_{d_k-1,k}}}\ne 0
$$
одночасно. Тоді функція $\tilde{F}$ є ніде недиференційовною на $[0;1]$.
\end{theorem}
\begin{proof}
Виберемо деяку нега-$(d_n)$-раціональну точку  $x_0$: 
$$
x_0=\Delta^{-(d_n)} _{\varepsilon_1\varepsilon_2...\varepsilon_{n-1}\varepsilon_n[d_{n+1}-1]0[d_{n+3}-1]...}=\Delta^{-(d_n)} _{\varepsilon_1\varepsilon_2...\varepsilon_{n-1}[\varepsilon_n-1]0[d_{n+2}-1]0[d_{n+4}-1]...}, 
$$
де $\varepsilon_n \ne 0$.

Введемо деякі позначення. Нехай $n$ --- непарне. Тоді
$$
x_0=x^{(1)}_0=\Delta^{-(d_n)} _{\varepsilon_1\varepsilon_2...\varepsilon_{n-1}\varepsilon_n[d_{n+1}-1]0[d_{n+3}-1]...}=
\Delta^{-(d_n)} _{\varepsilon_1\varepsilon_2...\varepsilon_{n-1}[\varepsilon_n-1]0[d_{n+2}-1]0[d_{n+4}-1]...}=x^{(2)}_0
$$
та у  випадку парного  $n$
$$
x_0=x^{(1)}_0=\Delta^{-(d_n)} _{\varepsilon_1\varepsilon_2...\varepsilon_{n-1}[\varepsilon_n-1]0[d_{n+2}-1]0[d_{n+4}-1]...}=
\Delta^{-(d_n)} _{\varepsilon_1\varepsilon_2...\varepsilon_{n-1}\varepsilon_n[d_{n+1}-1]0[d_{n+3}-1]...}=x^{(2)}_0.
$$

Розглянемо числові послідовності $(x^{'} _k)$, $(x^{''} _k)$:  
$$
x^{'} _k=\begin{cases}
\Delta^{-(d_n)} _{\varepsilon_1...\varepsilon_{n-1}\varepsilon_n[d_{n+1}-1]0[d_{n+3}-1]0...[d_{n+k-1}-1]1[d_{n+k+1}-1]0[d_{n+k+3}-1]...},&\text{$n$ --- непарне, $k$ --- парне;}\\
\Delta^{-(d_n)} _{\varepsilon_1...\varepsilon_{n-1}\varepsilon_n[d_{n+1}-1]0...[d_{n+k-2}-1]0[d_{n+k}-2]0[d_{n+k+2}-1]0[d_{n+k+4}-1]...},&\text{$n$ --- непарне, $k$ --- непарне;}\\
\Delta^{-(d_n)} _{\varepsilon_1...\varepsilon_{n-1}[\varepsilon_n-1]0[d_{n+2}-1]0[d_{n+4}-1]0...[d_{n+k-1}-1]1[d_{n+k+1}-1]0[d_{n+k+3}-1]...},&\text{  $n$ --- парне, $k$ --- непарне;}\\
\Delta^{-(d_n)} _{\varepsilon_1...\varepsilon_{n-1}[\varepsilon_n-1]0[d_{n+2}-1]0...[d_{n+k-2}-1]0[d_{n+k}-2]0[d_{n+k+2}-1]0[d_{n+k+4}-1]...},&\text{$n$ --- парне, $k$ --- парне,}
\end{cases}
$$
$$
x^{''} _k=\begin{cases}
\Delta^{-(d_n)} _{\varepsilon_1...\varepsilon_{n-1}[\varepsilon_n-1]0[d_{n+2}-1]0...[d_{n+k-1}-1]00[d_{n+k+2}-1]0[d_{n+k+4}-1]...},&\text{$n$ --- непарне, $k$ --- непарне;}\\
\Delta^{-(d_n)} _{\varepsilon_1...\varepsilon_{n-1}[\varepsilon_n-1]0[d_{n+2}-1]0...[d_{n+k}-1][d_{n+k+1}-1]0[d_{n+k+3}-1]0[d_{n+k+5}-1]...},&\text{$n$ --- непарне, $k$ --- парне;}\\
\Delta^{-(d_n)} _{\varepsilon_1...\varepsilon_{n-1}\varepsilon_n[d_{n+1}-1]0[d_{n+3}-1]...0[d_{n+k}-1][d_{n+k+1}-1]0[d_{n+k+3}-1]...},&\text{$n$ --- парне, $k$ --- непарне;}\\
\Delta^{-(d_n)} _{\varepsilon_1...\varepsilon_{n-1}\varepsilon_n[d_{n+1}-1]0...[d_{n+k-1}-1]00[d_{n+k+2}-1]0[d_{n+k+4}-1]0[d_{n+k+6}-1]...},&\text{$n$ --- парне, $k$ --- парне,}
\end{cases}
$$
Тобто,
$$
x^{'} _k=x^{(1)} _0+\frac{1}{d_1d_2...d_{n+k}},
$$
$$
x^{''} _k=x^{(2)} _0-\frac{1}{d_1d_2...d_{n+k}}
$$
та $x^{'} _k \to x_0$, $x^{''} _k \to x_0$ при $k \to \infty$.

Нехай  $n$ --- непарне. Тоді
$$
y^{(1)} _{0}=g(x^{(1)} _0)=\Delta^{D} _{\varepsilon_1[d_2-1-\varepsilon_2]\varepsilon_3[d_4-1-\varepsilon_4]\varepsilon_5...[d_{n-1}-1-\varepsilon_{n-1}]\varepsilon_n(0)},
$$
$$
y^{(2)} _{0}=g(x^{(2)} _0)=\Delta^{D} _{\varepsilon_1[d_2-1-\varepsilon_2]\varepsilon_3[d_4-1-\varepsilon_4]...[d_{n-1}-1-\varepsilon_{n-1}][\varepsilon_n-1][d_{n+1}-1][d_{n+2}-1][d_{n+3}-1]...},
$$
$$
y^{'} _k=g(x^{'} _k)=\Delta^D _{\varepsilon_1[d_2-1-\varepsilon_2]\varepsilon_3[d_4-1-\varepsilon_4]...[d_{n-1}-1-\varepsilon_{n-1}]\varepsilon_n\underbrace{0...0}_{k-1}1(0)},
$$
$$
y^{''} _k=g(x^{''} _k)=\Delta^D _{\varepsilon_1[d_2-1-\varepsilon_2]\varepsilon_3[d_4-1-\varepsilon_4]...[d_{n-1}-1-\varepsilon_{n-1}][\varepsilon_n-1][d_{n+1}-1][d_{n+2}-1]...[d_{n+k}-1](0)},
$$
де (як згадувалося вище) $\tilde{F}(x)=F(g(x))=F\circ g$.

Тоді
$$
\tilde F(x^{'} _k)=F(y^{'} _k)=\beta_{\varepsilon_1,1}+\sum^{n-1} _{t=2} {\left(\tilde \beta_{\varepsilon_t,t}\prod^{t-1} _{j=1} {\tilde p_{\varepsilon_j,j}}\right)}+\beta_{\varepsilon_n,n}\prod^{n-1} _{j=1} {\tilde p_{\varepsilon_j,j}}+
$$
$$
+\left(\sum^{n+k-1} _{l=n+1} {\left(\beta_{0,l}\prod^{l-1} _{m=n+1} {p_{0,m}}\right)}\right)\cdot\left(\prod^{n} _{j=1} {\tilde p_{\varepsilon_j,j}}\right)+\beta_{1,n+k}\left(\prod^{n} _{j=1} {\tilde p_{\varepsilon_j,j}}\right)\cdot \left(\prod^{n+k-1} _{m=n+1} {p_{0,m}}\right),
$$
$$
\tilde F(x^{(1)} _0)=F(y^{(1)} _0)=\beta_{\varepsilon_1,1}+\sum^{n-1} _{t=2} {\left(\tilde \beta_{\varepsilon_t,t}\prod^{t-1} _{j=1} {\tilde p_{\varepsilon_j,j}}\right)}+\beta_{\varepsilon_n,n}\prod^{n-1} _{j=1} {\tilde p_{\varepsilon_j,j}}.
$$

Отже,
$$
\tilde F(x^{'} _k)-\tilde F(x^{(1)} _0)=\beta_{1,n+k}\left(\prod^{n} _{j=1} {\tilde p_{\varepsilon_j,j}}\right)\cdot \left(\prod^{n+k-1} _{m=n+1} {p_{0,m}}\right)=\left(\prod^{n} _{j=1} {\tilde p_{\varepsilon_j,j}}\right)\cdot \left(\prod^{n+k} _{m=n+1} {p_{0,m}}\right).
$$

В свою чергу
$$
\tilde F(x^{(2)} _0)=F(y^{(2)} _0)=\beta_{\varepsilon_1,1}+\sum^{n-1} _{t=2} {\left(\tilde \beta_{\varepsilon_t,t}\prod^{t-1} _{j=1} {\tilde p_{\varepsilon_j,j}}\right)}+\beta_{\varepsilon_n-1,n}\prod^{n-1} _{j=1} {\tilde p_{\varepsilon_j,j}}+ 
$$
$$
+p_{\varepsilon_n-1,n}\left(\prod^{n-1} _{j=1} {\tilde p_{\varepsilon_j,j}}\right)\left(\sum^{\infty} _{l=n+1}{\left[\beta_{d_l-1,l}\prod^{l-1} _{m=n+1} {p_{d_m-1,m}}\right]}\right),
$$
$$
\tilde F(x^{''} _k)=F(y^{''} _k)=\beta_{\varepsilon_1,1}+\sum^{n-1} _{t=2} {\left(\tilde \beta_{\varepsilon_t,t}\prod^{t-1} _{j=1} {\tilde p_{\varepsilon_j,j}}\right)}+\beta_{\varepsilon_n-1,n}\prod^{n-1} _{j=1} {\tilde p_{\varepsilon_j,j}}+ 
$$
$$
+p_{\varepsilon_n-1,n}\left(\prod^{n-1} _{j=1} {\tilde p_{\varepsilon_j,j}}\right)\left(\sum^{n+k} _{l=n+1}{\left[\beta_{d_l-1,l}\prod^{l-1} _{m=n+1} {p_{d_m-1,m}}\right]}\right).
$$
Звідки
$$
\tilde F(x^{(2)} _0)-\tilde F(x^{''} _k)=p_{\varepsilon_n-1,n}\left(\prod^{n-1} _{j=1} {\tilde p_{\varepsilon_j,j}}\right)\cdot \left(\prod^{n+k} _{m=n+1} {p_{d_m-1,m}}\right).
$$

Перейдемо до розгляду випадку, коли  $n$ --- парне. В такому разі
$$
y^{(1)} _{0}=g(x^{(1)} _0)=\Delta^{D} _{\varepsilon_1[d_2-1-\varepsilon_2]\varepsilon_3[d_4-1-\varepsilon_4]...\varepsilon_{n-1}[d_{n}-\varepsilon_{n}](0)},
$$
$$
y^{(2)} _{0}=g(x^{(2)} _0)=\Delta^{D} _{\varepsilon_1[d_2-1-\varepsilon_2]\varepsilon_3[d_4-1-\varepsilon_4]...\varepsilon_{n-1}[d_n-\varepsilon_n-1][d_{n+1}-1][d_{n+2}-1]...},
$$
$$
y^{'} _k=g(x^{'} _k)=\Delta^D _{\varepsilon_1[d_2-1-\varepsilon_2]\varepsilon_3[d_4-1-\varepsilon_4]...\varepsilon_{n-1}[d_n-\varepsilon_n]\underbrace{0...0}_{k-1}1(0)},
$$
$$
y^{''} _k=g(x^{''} _k)=\Delta^D _{\varepsilon_1[d_2-1-\varepsilon_2]\varepsilon_3[d_4-1-\varepsilon_4]...\varepsilon_{n-1}[d_n-1-\varepsilon_n][d_{n+1}-1][d_{n+2}-1]...[d_{n+k}-1](0)},
$$

Як наслідок,
$$
\tilde F(x^{'} _k)=F(y^{'} _k)=\beta_{\varepsilon_1,1}+\sum^{n-1} _{t=2} {\left(\tilde \beta_{\varepsilon_t,t}\prod^{t-1} _{j=1} {\tilde p_{\varepsilon_j,j}}\right)}+\beta_{d_n-\varepsilon_n,n}\prod^{n-1} _{j=1} {\tilde p_{\varepsilon_j,j}}+
$$
$$
+\beta_{1,n+k}\left(\prod^{n-1} _{j=1} {\tilde p_{\varepsilon_j,j}}\right)\cdot \left(\prod^{n+k-1} _{m=n+1} {p_{0,m}}\right)p_{d_n-\varepsilon_n,n},
$$
$$
\tilde F(x^{(1)} _0)=F(y^{(1)} _0)=\beta_{\varepsilon_1,1}+\sum^{n-1} _{t=2} {\left(\tilde \beta_{\varepsilon_t,t}\prod^{t-1} _{j=1} {\tilde p_{\varepsilon_j,j}}\right)}+\beta_{d_n-\varepsilon_n,n}\prod^{n-1} _{j=1} {\tilde p_{\varepsilon_j,j}}.
$$

Отже,
$$
\tilde F(x^{'} _k)-\tilde F(x^{(1)} _0)=\beta_{1,n+k}\left(\prod^{n-1} _{j=1} {\tilde p_{\varepsilon_j,j}}\right)\cdot \left(\prod^{n+k-1} _{m=n+1} {p_{0,m}}\right)p_{d_n-\varepsilon_n,n}=
$$
$$
=\left(\prod^{n-1} _{j=1} {\tilde p_{\varepsilon_j,j}}\right)\cdot \left(\prod^{n+k} _{m=n+1} {p_{0,m}}\right)p_{d_n-\varepsilon_n,n}.
$$

В свою чергу
$$
\tilde F(x^{(2)} _0)=F(y^{(2)} _0)=\beta_{\varepsilon_1,1}+\sum^{n-1} _{t=2} {\left(\tilde \beta_{\varepsilon_t,t}\prod^{t-1} _{j=1} {\tilde p_{\varepsilon_j,j}}\right)}+\beta_{d_n-1-\varepsilon_n,n}\prod^{n-1} _{j=1} {\tilde p_{\varepsilon_j,j}}+ 
$$
$$
+\left(\prod^{n} _{j=1} {\tilde p_{\varepsilon_j,j}}\right)\left(\sum^{\infty} _{l=n+1}{\left[\beta_{d_l-1,l}\prod^{l-1} _{m=n+1} {p_{d_m-1,m}}\right]}\right),
$$
$$
\tilde F(x^{''} _k)=F(y^{''} _k)=\beta_{\varepsilon_1,1}+\sum^{n-1} _{t=2} {\left(\tilde \beta_{\varepsilon_t,t}\prod^{t-1} _{j=1} {\tilde p_{\varepsilon_j,j}}\right)}+\beta_{d_n-1-\varepsilon_n,n}\prod^{n-1} _{j=1} {\tilde p_{\varepsilon_j,j}}+ 
$$
$$
+\left(\prod^{n} _{j=1} {\tilde p_{\varepsilon_j,j}}\right)\left(\sum^{n+k} _{l=n+1}{\left[\beta_{d_l-1,l}\prod^{l-1} _{m=n+1} {p_{d_m-1,m}}\right]}\right).
$$
Звідки
$$
\tilde F(x^{(2)} _0)-\tilde F(x^{''} _k)=p_{d_n-1-\varepsilon_n,n}\left(\prod^{n-1} _{j=1} {\tilde p_{\varepsilon_j,j}}\right)\cdot \left(\prod^{n+k} _{m=n+1} {p_{d_m-1,m}}\right).
$$

Таким чином, 

$$
B^{'} _k=\frac{\tilde F(x^{'} _k)-\tilde F(x_0)}{x^{'} _k-x_0}=\begin{cases}
(d_n p_{\varepsilon_n,n})\left(\prod^{n-1} _{j=1} {d_j\tilde p_{\varepsilon_j,j}}\right) \left(\prod^{n+k} _{m=n+1} {d_mp_{0,m}}\right),\text{ $n$ --- непарне;}\\
(d_n p_{d_n-\varepsilon_n,n})\left(\prod^{n-1} _{j=1} {d_j\tilde p_{\varepsilon_j,j}}\right) \left(\prod^{n+k} _{m=n+1} {d_mp_{0,m}}\right),\text{ $n$ --- парне.}
\end{cases}
$$
$$
B^{''} _k=\frac{\tilde F(x_0)-\tilde F(x^{''} _k)}{x_0-x^{''} _k}=\begin{cases}
(d_n p_{\varepsilon_n-1,n})\left(\prod^{n-1} _{j=1} {d_j\tilde p_{\varepsilon_j,j}}\right) \left(\prod^{n+k} _{m=n+1} {d_m p_{d_m-1,m}}\right),\text{ $n$ --- непарне;}\\
(d_n p_{d_n-1-\varepsilon_n,n})\left(\prod^{n-1} _{j=1} {d_j\tilde p_{\varepsilon_j,j}}\right) \left(\prod^{n+k} _{m=n+1} {d_m p_{d_m-1,m}}\right),\text{$n$ --- парне.}
\end{cases}
$$

 Позначимо $b_{0,k}=\prod^{n+k} _{m=n+1} {d_mp_{0,m}}$ і  $b_{d_k-1,k}=~\prod^{n+k} _{m=n+1} {d_m p_{d_m-1,m}}$.

Оскільки, $\prod^{n-1} _{j=1} {d_j\tilde p_{\varepsilon_j,j}}=const$, $p_{\varepsilon_n,n}p_{\varepsilon_n-1,n}<0$,  $p_{d_n-\varepsilon_n,n}p_{d_n-1-\varepsilon_n,n}<0$ та за умовою теореми  послідовності $(b_{0,k})$,  $(b_{d_k-1,k})$  не збігаються до $0$ одночасно, отримаємо наступні випадки: 
\begin{enumerate}
\item якщо для всіх $k\in \mathbb N$ за винятком, можливо, скінченної множини номерів $k$: $d_kp_{0,k}>1$ та  $d_kp_{d_k-1,k}>1$, то одна з послідовностей  $B^{'} _k$, $B^{''} _k$ прямує до $\infty$, а інша --- до $- \infty$;

\item якщо для всіх $k\in \mathbb N$ за винятком, можливо, скінченної множини номерів $k$: один з добутків $d_kp_{0,k}$, $d_kp_{d_k-1,k}$ є більшим $1$, а інший --- меншим $1$, тоді одна з послідовностей  $B^{'} _k$, $B^{''} _k$ прямує до $\pm \infty$, а інша --- до $0$;

\item якщо для всіх $k\in \mathbb N$ за винятком, можливо, скінченної множини номерів $k$: один з добутків $d_kp_{0,k}$, $d_kp_{d_k-1,k}$ є більшим $1$, а інший --- рівним $1$, тоді одна з послідовностей  $B^{'} _k$, $B^{''} _k$ прямує до $\pm \infty$, а інша є сталою послідовністю;

\item якщо для всіх $k\in \mathbb N$ за винятком, можливо, скінченної множини номерів $k$: один з добутків $d_kp_{0,k}$, $d_kp_{d_k-1,k}$ є меншим $1$, а інший --- рівним $1$, тоді одна з послідовностей  $B^{'} _k$, $B^{''} _k$ прямує до $0$, а інша є сталою послідовністю;

\item якщо для всіх $k\in \mathbb N$  добутки $d_kp_{0,k}$, $d_kp_{d_k-1,k}$ є рівними $1$, то послідовності $B^{'} _k$, $B^{''} _k$ є різними сталими послідовностями,  оскільки $p_{\varepsilon_n,n} \ne p_{\varepsilon_n-1,n}$,  $p_{d_n-\varepsilon_n,n} \ne~p_{d_n-1-\varepsilon_n,n}$  в силу умов $p_{\varepsilon_k,k} \in (-1;1)$ та $\beta_{\varepsilon_k,k}>0$ для $\varepsilon_k>0$.
\end{enumerate}

Таким чином, функція $\tilde F$ є ніде не диференційовною на $[0;1]$, оскільки  у всіх випадках $\lim_{k \to \infty}{B^{'} _k} \ne \lim_{k \to \infty}{B^{''} _k}$.
\end{proof}


\begin{thebibliography}{9}

\bibitem{BPP09} \emph{Барановський~О.~М., Працьовита~І.~М., Працьовитий~М.~В.} Про одну функцію, пов'язану з рядами Остроградського 1-го та 2-го видів
  //~Науковий часопис НПУ імені М. П. Драгоманова. Серія 1. Фізико-математичні науки. "--- Київ: НПУ імені М. П. Драгоманова."--- 2009, № 10." --- С.~40~---~49.

\bibitem{Pra98} \emph{Працьовитий~М.~В.} Фрактальний підхід у дослідженнях сингулярних розподілів. "--- Київ: Вид-во НПУ імені М. П. Драгоманова,
  1998. "--- 296~с.
  
  \bibitem{PK2011} \emph{Працьовитий~М.~В., Калашніков~А.~В.} Про один клас неперервних функцій зі складною локальною будовою, більшість з яких сингулярні або недиференційовні
  //~Труды Ин-та прикл. математики и механики НАН Украины. ---  2011. --- № 23. --- С.~178 ---189.

\bibitem{Prats02} {\it Працьовитий М. В.}
Фрактальні властивості однієї неперервної ніде не диференційовної функції~// Наукові записки НПУ імені М.  П. Драгоманова. Фізико-математичні науки. --- Київ: НПУ імені М. П. Драгоманова. --- 2002, №3. --- С. 351-362.

\bibitem{Prats89} {\it Працевитый Н. В.} Непрерывные канторовские проекторы// Методы исследования алгебраических и топологических структур. --- К.: КГПИ. --- 1989.  --- С. 95-105.

\bibitem{Symon4} {\it Сербенюк С. О.}
Про деякi множини дiйсних чисел, визначенi в термiнах нега-s-кового та
канторiвського нега-s-кового зображень// Науковий часопис НПУ імені М. П. Драгоманова. Серія 1. Фізико-математичні науки.~--- Київ: НПУ імені М. П. Драгоманова. --- 2013, №15. --- С.~168-187. 
  
  \bibitem{Symon3} {\it Сербенюк С. О.}
Зображення чисел знакододатними рядами Кантора: задання рацiональних чисел// Науковий часопис НПУ імені М. П. Драгоманова. Серія 1. Фізико-математичні науки.~--- Київ: НПУ імені М. П. Драгоманова. --- 2013, №14. --- С.~253-267.

\bibitem{Symon12(2)} {\it Сербенюк С. О.}
Про одну майже скрiзь неперервну i нiде не диференцiйовну функцiю, яка задана автоматом зi скiнченною пам’яттю // Науковий часопис НПУ імені М. П. Драгоманова. Серія~1. Фізико-математичні науки. --- Київ: НПУ імені М. П. Драгоманова. --- 2012, №13(2). --- С. 166-182.

 \bibitem{Symon2015} {\it Сербенюк С. О.} Функції, означені системами функціональних рівнянь у термінах зображення чисел рядами Кантора // Науковi записки НаУКМА. --- 2015. --- Т. 165: Фізико-математичні науки.~--- С. 34-40.

\bibitem{Pra92} \emph{Турбин А. Ф., Працевитый~Н.~В., } Фрактальные множества, функции, распределения. "--- Киев: Наукова думка, 1992. "---  208~с.
 
\bibitem{Cantor1} \emph{Cantor~G.} Ueber die einfachen Zahlensysteme
  //~Z. Mathl. Phys. "--- 1869."--- Bd. 14."--- S.~121--128.
  


\end{thebibliography}
\end{document}